\documentclass[tombow,center]{aspm}

\articleinfo{Stochastic Analysis, Random Fields and Integrable Probability -- Fukuoka 2019}{87}{2021}

\usepackage{acronym}
\acrodef{KPZ}{Kardar--Parisi--Zhang}
\acrodef{SHE}{Stochastic Heat Equation}
\acrodef{LDP}{Large Deviation Principle}
\acrodef{GUE}{Gaussian Unitary Ensemble}
\acrodef{BM}{Brownian Motion}
\renewcommand{\Pr}{\mathbf{P}}	% probability
\newcommand{\Ex}{\mathbf{E}}	% expectation
\renewcommand{\d}{\mathrm{d}}	% differential

\newcommand{\ind}{\mathbf{1}}	%indicator function
\newcommand{\Ai}{\mathrm{Ai}}	% Airy function
\newcommand{\aip}{\mathbf{a}}	% Airy PP
\newcommand*{\Cdot}{{\raisebox{-0.5ex}{\scalebox{1.8}{$\cdot$}}}} % largedot
\newcommand{\img}{\mathbf{i}}	% \sqrt{-1}

\newcommand{\sao}{\mathcal{A}}	% stochastic Airy operator
\newcommand{\saolambda}{\boldsymbol{\lambda}} %eigenvalues of sao
\newcommand{\saolambdaa}{\boldsymbol{\eta}} %eigenvalues of sao
\newcommand{\hill}{\mathcal{H}}	% Hill's operator
\newcommand{\hilllambda}{\boldsymbol{\lambda}} %eigenvalues of hill

\newcommand{\R}{\mathbb{R}}
\newcommand{\Z}{\mathbb{Z}}

\newcommand{\til}{\widetilde}
\renewcommand{\hat}{\widehat}

\allowdisplaybreaks
\usepackage{graphicx}	%for including graphs
\usepackage{verbatim}
\usepackage{amssymb}
\usepackage{amsbsy}
\usepackage{amscd}
\usepackage{amsmath}
\usepackage{amsthm}
\usepackage[mathscr]{eucal}

\newtheorem{theorem}{Theorem}
\newtheorem{prop}[theorem]{Proposition}

\title[LDs of the KPZ equation via the Stochastic Airy Operator]{Large deviations of the KPZ equation\\{}via the Stochastic Airy Operator}

\author[Li-Cheng Tsai]{Li-Cheng Tsai}

\address{Department of Mathematics, Rutgers University --- New Brunswick\\{}%
110 Frelinghuysen Road, Piscataway, NJ 08854 USA}

\email{lctsai.math@gmail.com}

\rcvdate{December 26, 2019}
\rvsdate{September 25, 2020}


\begin{document}

\begin{abstract}
In this article we review the ideas in \cite{tsai18} toward proving the one-point, lower-tail large deviation principle for the  Kardar--Parisi--Zhang equation.
\end{abstract}

\maketitle

\section{Introduction}

The \ac{KPZ} equation was introduced in \cite{kardar86} as a model of random surface growth.
In one spatial dimensional the equation reads
\begin{align}
	\label{e.kpz}
	\partial_t h = \tfrac12 \partial_{xx} h + \tfrac12 (\partial_x h )^2 + \xi,
\end{align}
where $ \xi=\xi(t,x) $ denotes the spacetime white noise, and the solution $ h=h(t,x) $ is a random function that describes the height at time $ t\in\R_+ $ and position $ x\in\R $. 
Together with a host of related models, the \ac{KPZ} equation has been intensively studied,
due to its rich connections to other physical phenomena and mathematical structures.
We refer to \cite{ferrari10,quastel11,corwin12,quastel15,chandra17,corwin19} for reviews on studies related to the \ac{KPZ} equation.

Due to the roughness of $ \xi $, the solution $ h $ is only $ a $-H\"{o}lder continuous in $ x $ for $ a<\frac12 $. This fact together with the presence of the nonlinear term $ (\partial_x h )^2 $ makes the \ac{KPZ} equation \eqref{e.kpz} ill-posed. New theories have been built toward making sense of the \ac{KPZ} equation and constructing the corresponding solution. We refer to \cite{hairer14,gubinelli15,gonccalves14,gubinelli18} and the references therein for related developments.
An alternative formulation to these theories is the \textbf{Hopf--Cole solution}. That is, a \emph{formal} exponentiation $ Z(t,x) := e^{h(t,x)} $ brings \eqref{e.kpz} to the \ac{SHE}
\begin{align}
	\label{e.she}
	\partial_t Z = \tfrac12 \partial_{xx} Z + \xi Z .
\end{align}
This equation is well-posed \cite{walsh86,bertini95} and the solution $ Z(t,x) $ is strictly positive for $ t>0 $ and for generic nonnegative and nonzero initial data \cite{mueller91,flores14}. These facts allow us to define the Hopf--Cole solution $ h(t,x) := \log Z(t,x) $. The Hopf--Cole formulation arises in several discrete or regularized versions of the \ac{KPZ} equation, and other notions of solutions from the aforementioned theories have been shown to coincide with the Hopf--Cole solution within the relevant class of initial data.

In this article we are concerned with the one-point, lower-tail \ac{LDP} for the \ac{KPZ} equation.
Consider the Hopf--Cole solution $ h(t,x) := \log Z(t,x) $ with the initial data $ Z(0,x) = \delta(x) $, a Dirac delta at the origin.
It is known that, for large time $ t\gg 1 $, the height $ h(2t,0) $ concentrates around $ -\frac{t}{12} $.
The question of interest here is to estimate the probability of $ h(2t,0) $ being much smaller than this typical value $ -\frac{t}{12} $.
This question has been much studied recently in the physics and mathematics communities.
In particular, the physics works \cite{sasorov17,corwin18,krajenbrink18} each employed a different method to derive the explicit rate function 
\begin{align*}
	\Pr\big[ h(2t,0) \leq -\tfrac{t}{12} + tz \big] \approx e^{-t^2\Phi_-(z)},
	\quad
	z<0,
\end{align*}
where
\begin{align}
	\label{e.rate}
	\Phi_-(z) := \tfrac{4}{15\pi^6}(1-\pi^2z)^{\frac52}-\tfrac{4}{15\pi^6}+\tfrac{2}{3\pi^4}z-\tfrac{1}{2\pi^2} z^2.
\end{align}
The first rigorous proof came soon later, via yet another method:
\begin{theorem}[\cite{tsai18}]
\label{t.main.}
Consider the Hopf--Cole solution $ h(t,x) := \log Z(t,x) $ of the \ac{KPZ} equation with the initial data $ Z(0,x)=\delta(x) $. For any $ z<0 $,
\begin{align*}
	\lim_{t\to\infty} \frac{1}{t^2} &\log \Pr\big[ h(2t,0) \leq -\tfrac{t}{12} + tz \big]
	=
	-\Phi_-(z).
\end{align*}
\end{theorem}

The four different methods \cite{sasorov17,corwin18,krajenbrink18,tsai18} were later shown to be closely related in \cite{krajenbrink19}.
Two new methods have been recently obtained, in the mathematically rigorous work \cite{cafasso19} and the physics work \cite{ledoussal19}.
This article focuses on reviewing the method used in~\cite{tsai18}. 

%\begin{rmk}
It is known that the large deviation rate function depends on the initial data,
and there  has been many recent results for general initial data or special initial data other than $ Z(0,x)=\delta(x) $.
In the mathematics literature,
the work \cite{corwin20general} proved upper- and lower-tail probability bounds for general initial data;
the work \cite{kim19} proved lower-tail probability bounds for the narrow-wedge initial data in the half-space geometry;
the work \cite{ghosal20} proved the one-point, upper-tail \ac{LDP} for general initial data;
the work \cite{lin20} proved the one-point, upper-tail \ac{LDP} for the narrow-wedge initial data in the half-space geometry.
For the physics literature we refer to \cite{krajenbrink19a} and the references therein.
%\end{rmk}

\subsection*{Acknowledgements.}
Research partially supported by the NSF through DMS-1712575.

\section{Exact formulas and the stochastic Airy operator}
Hereafter $ Z(t,x) $ denotes the solution of the \ac{SHE} with $ Z(0,x)=\delta(x) $ and $ h(t,x) := \log Z(t,x) $.
We recall the formula that expresses the Laplace transform of $ Z(2t,0) $ in terms of a Fredholm determinant:
\begin{align}
	\label{e.fred.det}
	\Ex\big[ e^{-sZ(2t,0)e^{\frac{t}{12}}} \big] = \det( I - K_{s,t} ),
	\quad
	s,t>0,
\end{align}
where $ K_{s,t} $ is a trace-class operator on $ L^2[0,\infty) $ with the integral kernel
$ K_{s,t}(x,y) := \int_{\R} \frac{\d r}{1+s^{-1}e^{-t^{1/3}r}} \Ai(x+r) \Ai(y+r) $, and $ \Ai $ denotes the Airy function.
The formula (or a closely related version of it) 
was derived simultaneously and independently in the works \cite{calabrese10,amir11,dotsenko10,sasamoto10}, and \cite{amir11} provided a rigorous proof.

The formula \eqref{e.fred.det} provides access to the distribution of $ h(2t,0) $,
for example, in deriving the Tracy--Widom fluctuation of $ h(2t,0) $ at large time \cite{calabrese10,amir11,dotsenko10,sasamoto10}.
Another instance is the upper-tail large deviations.
For $ s=e^{-zt} $ and $ z>0 $, the determinant in~\eqref{e.fred.det} behaves perturbatively,
and (with some modifications) can be used to derive the upper-tail large deviations of $ h(2t,0) $ \cite{ledoussal16,das19}.

In the lower-tail regime considered here, the determinant in~\eqref{e.fred.det} does not provide a handy access to the \ac{LDP}. 
Instead, we appeal to a different expression of the formula from \cite{borodin16}:
\begin{align}
	\label{e.bg}
	\Ex\big[ e^{-sZ(2t,0)e^{\frac{t}{12}}} \big] 
	=
	\Ex\Big[ \prod_{i=1}^\infty \frac{1}{1+se^{t^{1/3}\aip_i}} \Big],
	\quad
	t,s>0.
\end{align}
On the r.h.s., the expectation is taken with respect to the \textbf{Airy point process}
$ -\infty<\ldots<\aip_3<\aip_2<\aip_1<\infty $, which is the determinantal point process on $ \R $ with the correlation kernel $ K_{\text{Airy PP}}(x,y) := \int_{\R_+}\d r \Ai(x+r)\Ai(y+r) $.
In~\eqref{e.bg}, substitute in $ s=e^{-tz} $ and $ Z(2t,0)=e^{h(2t,0)} $. We rewrite the formula as
\begin{align}
	\label{e.bg.}
	\Ex\big[ F(h(2t,0)+\tfrac{t}{12}-tz) \big] 
	=
	\Ex\Big[ \exp\Big(-t \int_{\R} \d \mu_{\aip,t}(a) \, \psi_{t,z}(a) \Big) \Big],
\end{align}
where $ F(x) := \exp(-e^x) $ and $ \psi_{t,z}(a) := \log(1+e^{-t(z+a)}) $,
and $ \mu_{\aip,t}(a) := t^{-1}\sum_{i=1}^\infty \delta_{-t^{-2/3}\aip_i}(a) $ denotes the empirical measure of the scaled, spaced-reversed Airy point process.

As noted in \cite{corwin18a}, the formulas \eqref{e.bg}--\eqref{e.bg.} provide the suitable framework for the lower-tail \ac{LDP}.
To see how, we discuss the left and right hand sides of \eqref{e.bg.}:
\begin{enumerate}
\item[(LHS)] \label{observation1}
	The function $ F(x) $ approaches $ 0 $ and $ 1 $ respectively as $ x\to\infty $ and as $ x\to -\infty $.
	Together with the $ t $ scaling, the function $ F(x) $ serves as a good proxy for $ \ind_{x<0} $, thereby
	\begin{align*}
		\text{ (l.h.s.\ of \eqref{e.bg.})}
		\approx
		\Pr\big[ h(2t,0)+\tfrac{t}{12}-tz \leq 0 \big].
	\end{align*}
	Note that this approximation holds even in the large deviation regime, because $ F(x)\to 0 $ \emph{super}-exponentially as $ x\to\infty $.
\item[(RHS)] \label{observation2}
	For $ t\gg 1 $, $ \psi_{t,z}(a) \approx t(z+a)_- $, where $ x_- := \max\{-x,0\} $ denotes the negative part of $ x $.
	Hence
	\begin{align*}
		\text{ (r.h.s.\ of \eqref{e.bg.})}
		\approx
		\Ex\Big[ \exp\Big( -t^2 \int_{\R} \d \mu_{\aip,t}(a) (a+z)_- \Big) \Big].
	\end{align*}
	Assuming that the random measure $ \mu_{\aip,t} $ enjoys an \ac{LDP} with speed $ t^2 $ and a rate function $ I_{\aip} $,
	we should have
	\begin{align*}
		\text{ (r.h.s.\ of \eqref{e.bg.})}
		\approx
		\exp\Big( -t^2 \inf_{\mu} \Big\{ \int_{\R} \d \mu(a) \,  (a+z)_- + I_\aip(\mu) \Big\}  \Big),
	\end{align*}
	where the infimum is taken over a suitable class of measures $ \mu $.
\item[(Var)] \label{observation3}
	Combining the preceding two observations, one expects
	\begin{align}
		\label{e.var}
		\Phi_-(z) = \inf_{\mu} \Big\{ \int_{\R} \d \mu(a) \,  (a+z)_- + I_\aip(\mu) \Big\}.
	\end{align}
\end{enumerate}

The observations (RHS)--(Var) were first made and noted in \cite{corwin18a}.
Based on these observations, \cite{corwin18a} obtained detailed bounds on the tail probability of $ h(2t,0) $.
The physics work \cite{corwin18} continued along this path.
It is known that the Airy point process $ \{\aip_i\}_{i=1}^\infty $ is the limit near the top edge of the spectrum of the \ac{GUE}.
Employing a non-rigorous limit transition from the known rate function of the \ac{GUE} \cite{benarous97},
the work \cite{corwin18} obtained a conjectural form of $ I_\aip $,
and solved the variational problem \eqref{e.var} to obtain the rate function \eqref{e.rate}.

The work \cite{tsai18} also proceeds through \eqref{e.var}, but, 
instead of viewing the Airy point process as a limit of the GUE, appeals to the stochastic Airy operator.
For $ \beta>0 $, consider the random operator
\begin{align}
	\label{e.sao}
	\sao_\beta = -\tfrac{\d^2~}{\d x^2} + x + \tfrac{2}{\sqrt\beta} B',
	\quad
	x\in [0,\infty),
\end{align}
where $ B' $ denotes the derivative of a \ac{BM}.
It is standard to construct $ \sao_\beta $ as an unbounded, self-adjoint operator on $ L^2[0,\infty) $ with the Dirichlet boundary condition at $ x=0 $,
and the so constructed operator has a pure-point, bounded below spectrum, $ -\infty<\saolambda_1<\saolambda_2<\saolambda_3<\ldots $.
This spectrum offers an alternative description of the Airy point process:

\begin{theorem}[\cite{ramirez11}]
The spectrum of $ \sao_2 $ is equal in law to the space-reversed Airy point process, i.e., $ \{\saolambda_i\}_{i=1}^\infty \stackrel{\text{law}}{=} \{-\aip_i\}_{i=1}^\infty $.
\end{theorem}

The main theorem of \cite{tsai18} can now be stated as:
\begin{theorem}[\cite{tsai18}]
\label{t.main}
For any $ \beta>0 $ and $ z<0 $, let $ \saolambda_i $, $ i=1,2,\ldots $, denote the eigenvalues of $ \sao_\beta $. We have
\begin{align}
	\label{e.t.main}
	\lim_{t\to\infty} \frac{1}{t^2} \log \Big( \Ex\big[ e^{ - \sum_{i=1}^\infty (\saolambda_i t^{1/3}+zt )_- } \big] \Big)
	=
	- \big(\tfrac{2}{\beta}\big)^5 \Phi_-\big( \big(\tfrac{\beta}{2}\big)^2 z \big).
\end{align}
\end{theorem}
\noindent%
Theorem~\ref{t.main.} then follows from Theorem~\ref{t.main} for $ \beta=2 $ and the preceding observations (LHS)--(RHS).

We devote the rest of the article to explaining the ideas of the proof of Theorem~\ref{t.main}.
The proof proceeds in steps, which are designated in the titles of the remaining sections.

\section{Localization via Riccati transform}
\label{s.local}
The stochastic Airy operator \eqref{e.sao} has a linear potential $ x $.
Such a potential is physically relevant because it ensures that $ \sao_\beta $, which acts on the unbounded interval $ [0,\infty) $, has a pure-point spectrum.
For our analysis, however, a varying potential is inconvenient.
We hence seek to approximate $ \sao_\beta $ by a sequence of operators with translation-invariant potentials.
Fix a mesoscopic scale $ \Xi = t^{a} $, where the value of $ a $ will be specified later in \eqref{e.a.range}.
For $ j\in\Z_{\geq 0} $, consider the (shifted) Hill's operator
\begin{align}
	\label{e.hill}
	\hill_j := -\tfrac{\d^2~}{\d y^2} + j\Xi + \tfrac{2}{\sqrt{\beta}} W',
	\
	y\in[0,\Xi],
\end{align}
with the Dirichlet boundary condition at $ y=0 $ and $\Xi $, where $ W $ denotes a \ac{BM}, and for different $ j $'s the operators $ \{\hill_j\}_{j=0}^\infty $ are independent.

The idea is that the term $ j\Xi $ in \eqref{e.hill} approximates the linear potential $ x $ for $ x=y+(j-1)\Xi $, on the interval $ y\in [0,\Xi] $ or $ x\in[(j-1)\Xi,j\Xi] $.
We then seek to `piece together' the operators $ \hill_1,\hill_2,\ldots $ to approximate $ \sao_\beta $.
Doing so requires the \textbf{Riccati transform}.
The transform begins with the eigenvalue problem $ \sao_\beta f = -f''+xf+\frac{2}{\sqrt\beta}B'f = \lambda f $, with the boundary condition $ f(0)=0 $.
Viewing this equation as a second-order ODE, we perform the transformation $ g=f'/f $ into a first-order ODE of the Riccati type:
$
	g'=x-\lambda-g^2+\tfrac{2}{\sqrt\beta}B'.
$
See Figure~\ref{f.riccati} for schematic graphs of $ f $ and $ g $.
With $ g=f'/f $, we see that $ g $ explodes to $ \pm\infty $ whenever $ f $ crosses zero.
Since the underlying space $ [0,\infty) $ is one-dimensional, the $ k $-th eigenvector has exactly $ k $ roots, 
and hence the function $ g $ explodes exactly $ k $ times, excluding the explosion at $ x=0 $.

\begin{figure}[h]
\begin{center}
\includegraphics[clip, keepaspectratio, bb=0 0 398 273, width=.75\textwidth]{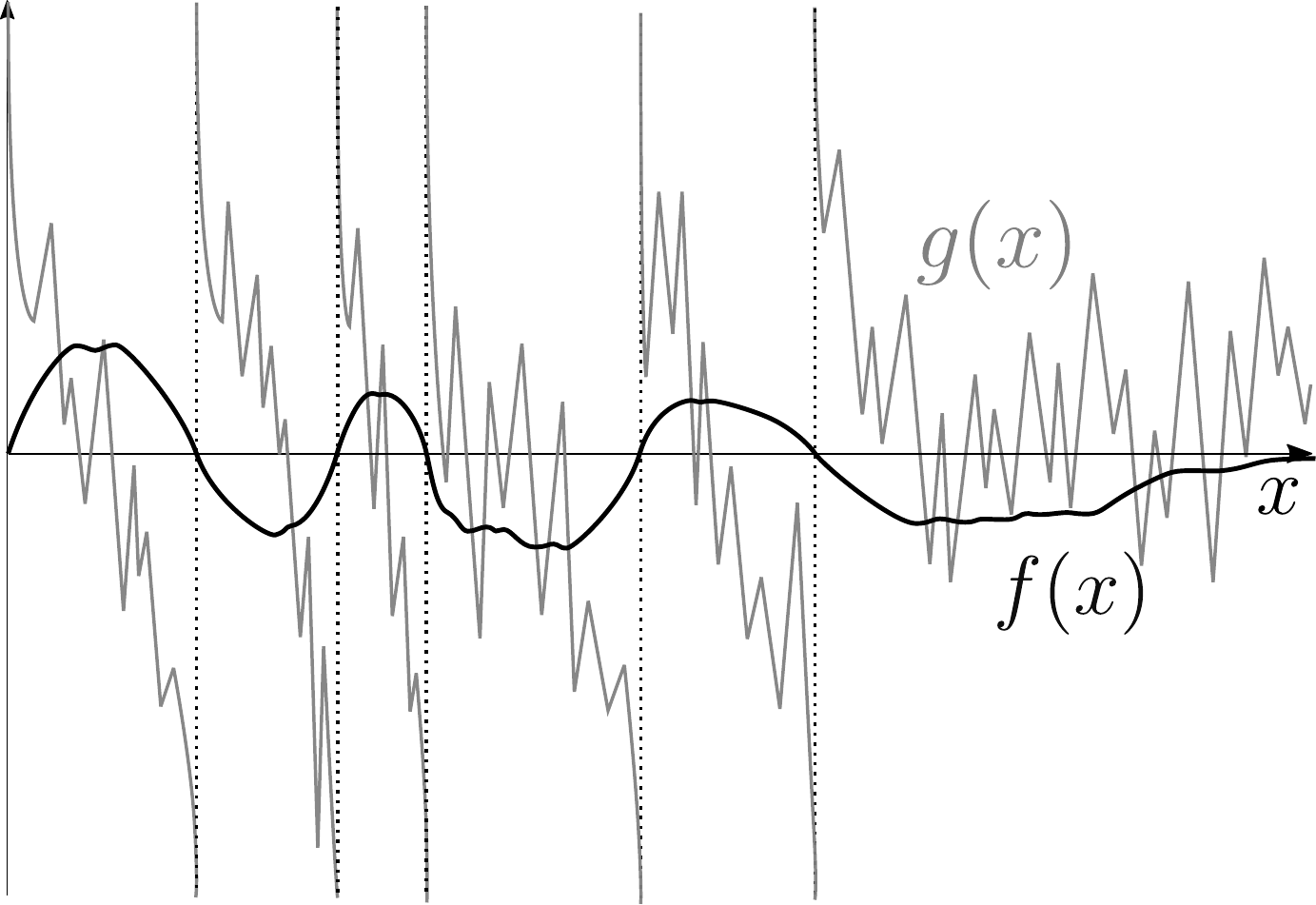}
\end{center}
\caption{Schematic graphs of $f $ and $ g $ in the Riccati transform}
\label{f.riccati}
\end{figure}
Now, let us view $ \lambda\in\R $ as an \emph{arbitrary} parameter, and solve the following ODE: 
\begin{align}
	\label{e.riccati}
	g'_\lambda(x) = x-\lambda - g^2_\lambda(x) + \tfrac{2}{\sqrt\beta}B'(x),
	\
	x\in(0,\infty),
	\
	g_\lambda(0) = +\infty.
\end{align}
The solution $ g_\lambda $ may undergo explosions to $ -\infty $, and whenever that happens we immediately reinitiate $ g_\lambda $ from $ +\infty $.
Let $ N(\lambda) := \#\{\saolambda_i\leq \lambda\} $ counts the eigenvalues of $ \sao_\beta $ at most $ \lambda $.

\begin{prop}[Prop.~3.4.~in \cite{ramirez11}]
\label{p.riccati}
Almost surely for all $ \lambda\in\R $,
$
	\#\{ x\in(0,\infty) : |g_\lambda(x)|=\infty \} = N(\lambda).
$
\end{prop}

We can also consider the analogous ODE for Hill's operator
\begin{align}
	\label{e.riccati.local}
	g'_{\lambda,j}(y) = j\Xi -\lambda - g^2_{\lambda,i}(y) + \tfrac{2}{\sqrt\beta}W'(y),
	\
	g_{\lambda,i}(0) = +\infty,
	\
	y\in(0,\Xi].
\end{align}
Let $ \{-\infty<\hilllambda^{(j)}_{1}\leq \hilllambda^{(j)}_{2} \leq \ldots\} $ denote the spectrum of $ \hill_j $,
and similarly $ N_j(\lambda) := \#\{\saolambda^{(j)}_i\leq \lambda\} $.
Similarly to Proposition~\ref{p.riccati}, $ N_j(\lambda) $ is equal to the number of explosions of \eqref{e.riccati.local}.

Having introduced the Riccati transforms for $ \sao_\beta $ and $ \hill_j $,
we now apply these transforms to compare the counting functions $ N(\lambda) $ and $ N_j(\lambda) $.
Localize \eqref{e.riccati} onto the interval $ x\in((j-1)\Xi,j\Xi] $,
and set $ x=y+(j-1)\Xi $ to get
\begin{align}
	\tag{\ref{e.riccati}'}
	\label{e.riccati.}
	g'_\lambda = (y+(j-1)\Xi))-\lambda - g^2_\lambda + \tfrac{2}{\sqrt\beta}B'.
\end{align}
We can now couple \eqref{e.riccati.} with \eqref{e.riccati.local} by $ W(y)=B(y+(j-1)\Xi) $.
On the interval $ y\in(0,\Xi] $ we see that \eqref{e.riccati.local} has a larger potential $ j\Xi \geq (y+(j-1)\Xi)) $
and a larger entrance value $ g_{\lambda,j}(0)=+\infty \geq g_{\lambda}((j-1)\Xi) $.
Comparison arguments then yield
\begin{align}
	\label{e.compare1}
	N_j(\lambda) \leq \#\{ x\in((j-1)\Xi,j\Xi] : |g_\lambda(x)|=\infty \}.
\end{align}
To get the reverse inequality, we apply the same coupling $ W(y)=B(y+(j-1)\Xi) $
for \eqref{e.riccati.local} with $ j\mapsto j-1 $ and for \eqref{e.riccati.} with $ j $. 
The potential is now reversely ordered $ (j-1)\Xi \leq (y+(j-1)\Xi)) $, though the entrance values are not.
The issue of entrance values can be cured by forgoing the first explosion of \eqref{e.riccati.} on $ y\in(0,\Xi] $, which gives
\begin{align}
	\label{e.compare2}
	N_{j-1}(\lambda)+1 \geq \#\{ x\in((j-1)\Xi,j\Xi] : |g_\lambda(x)|=\infty \}.
\end{align}

Next, to make use of the comparison results \eqref{e.compare1}--\eqref{e.compare2},
we rewrite the quantity of interest in Theorem~\ref{t.main} as
\begin{align}
	\label{e.reprint}
	- \sum_{i=1}^\infty (\saolambda_it^{1/3}+zt)_- 
	=
	- \hspace{-22pt}\int\limits_{(-\infty,-zt^{2/3}]} \hspace{-22pt} \d N(\lambda) \, (t^{1/3}\lambda+z)_-
	=
	-t^{1/3}\hspace{-22pt}\int\limits_{(-\infty,-zt^{2/3}]} \hspace{-22pt} \d\lambda \, N(\lambda),
\end{align}
and similarly
\begin{align}
	\label{e.reprint.}
	- \sum_{i=1}^\infty (\saolambda^{(j)}_it^{1/3}+zt)_- 
	=
	-t^{1/3}\hspace{-20pt}\int\limits_{(-\infty,-zt^{2/3}]} \hspace{-20pt} \d\lambda \, N_j(\lambda).
\end{align}
Combining~\eqref{e.compare1} and \eqref{e.reprint} gives, for any $ n\in\Z_{>0} $,
\begin{align*}
	- \sum_{i=1}^\infty (\saolambda_it^{1/3}+zt)_- 
	&\leq
	- \sum_{j=1}^n \sum_{i=1}^\infty (\saolambda^{(j)}_it^{1/3}+zt)_-.
\end{align*}
Here we forgo the explosions of \eqref{e.riccati} on $ [n\Xi,\infty) $ because they contribute negatively.
For the reverse inequality, we need to account for these remaining explosions.
To this end, consider the operator $ \sao_{\beta,n} := (-\frac{\d^2~}{\d x^2} + x + \frac{2}{\sqrt{\beta}}B') $ acting on $ x\in[n\Xi,\infty) $, with the Dirichlet boundary condition at $ x=n\Xi $.
We make $ (\sao_{\beta}+n \Xi) $ independent of $ \{\hill_{j}\}_{j=0}^{n-1} $.
Let $ \{\saolambdaa_i\}_{i=1}^\infty $ denote the spectrum of $ \sao_{\beta,n} $.
We have
\begin{align*}
	- \sum_{i=1}^\infty (\saolambda_it^{1/3}+zt)_- 
	&\geq
	-n - \sum_{j=0}^{n-1} \sum_{i=1}^\infty (\saolambda^{(j)}_it^{1/3}+zt)_- 
	- 
	\sum_{i=1}^\infty (\saolambdaa_it^{1/3}+zt)_-. 
\end{align*}
Exponentiate the preceding inequalities, take expectation,
and utilize the independence of $ \{\hill_i\}_{i=1}^n $ and of $ \{\hill_i\}_{i=0}^{n-1}, \sao_{\beta,n} $
to split the resulting expectations into products.
Note that $ \sao_{\beta,n} \stackrel{\text{law}}{=} \sao_{\beta}+n\Xi $.
We have 
\begin{prop}
\label{p.local}
For any $ n\in\Z_{>0} $, 
\begin{subequations}
\begin{align}
	\label{e.local.low}
	&\prod_{j=0}^{n-1} \Ex\big[ e^{ - \sum_{i=1}^\infty (\saolambda^{(j)}_it^{1/3}+zt)_- } \big]
	\cdot
	e^{-n}
	\cdot
	\Ex\big[ e^{ - \sum_{i=1}^\infty ((\saolambda_i+n\Xi)t^{1/3}+zt)_- } \big]
\\
	\label{e.local.up}
	&\leq
	\Ex\big[ e^{ - \sum_{i=1}^\infty (\saolambda_it^{1/3}+zt)_- } \big] \leq \prod_{j=1}^n \Ex\big[ e^{ - \sum_{i=1}^\infty (\saolambda^{(j)}_it^{1/3}+zt)_- } \big].
\end{align}
\end{subequations}
\end{prop}

\section{Lower bound}
\label{s.lower}
Proposition~\ref{p.local} reduces the problem of analyzing $ \sao_\beta $ to analyzing each $ \hill_j $.
Based on this reduction, we will explain how to obtain the desired lower and upper bounds 
on the quantity of interest, i.e., the l.h.s.\ of \eqref{e.t.main} in Theorem~\ref{t.main}.

Let us begin by fixing the scale $ \Xi=t^a $.
Referring to the l.h.s.\ of \eqref{e.t.main}, since $ z<0 $ is fixed, we see that the relevant eigenvalues should be of order $ t^{2/3} $.
Refer to \eqref{e.riccati};
if we also match the order of $ \lambda $ to $ t^{2/3} $, then the linear potential should vary at scale $ t^{2/3} $.
This observation forces us to choose $ a<2/3 $, since otherwise we cannot expect $ x $ to be approximated by a constant on the interval $ ((j-1)\Xi,j\Xi] $.
On the other hand, we wish $ \Xi=t^a $ to be greater than the time scale between explosions of \eqref{e.riccati}.
Doing so gives us some room for analysis.
Performing scaling in \eqref{e.riccati} under the assumptions that $ \lambda,x \asymp t^{2/3} $ shows that explosions of the ODE should occur at scale $ t^{-1/3} $.
Hence we require $ a>-\frac13 $.
It turns out that our analysis does not require any further condition on $ a $, and we hereafter fix
\begin{align}
	\label{e.a.range}
	\Xi = t^{a},
	\quad
	a \in (-\tfrac13, \tfrac23).
\end{align}

We now return to the task of obtaining the lower bound.
The expression \eqref{e.local.low} contains three terms.
Let us focus on the first term and show that, for some suitable $ n $ given later in \eqref{e.n},
%A careful analysis shows that, for $ n= -z t^{2/3-a} $, the last expectation in~\eqref{e.local.low} is negligible after being taken $ t^{-2}\log(\Cdot) $ and passed to the limit $ t\to\infty $.
%We do not perform such an analysis here. 
%Also, the term $ e^{-n} $ in \eqref{e.local.low}  is also negligible because $ t^{-2}\log(e^{-n}) = -t^{-2}(\log(-z)+t^{2/3-a}) \to 0 $ by~\eqref{e.a.range}.
%These facts reduce our goal to showing
\begin{align}
	\label{e.lower.goal}
	\liminf_{t\to\infty} \frac{1}{t^2} \log \Big( \prod_{j=0}^{n-1}\Ex\big[ e^{ - \sum_{i=1}^\infty (\saolambda^{(j)}_it^{1/3}+zt)_- } \big]\Big)
	\geq
	\big(\tfrac{2}{\beta}\big)^5 \Phi_-\big( \big(\tfrac{\beta}{2}\big)^2 z \big).
\end{align}
Once this is done, we will argue that the remaining two terms in \eqref{e.local.low} are negligible.

Recall from \eqref{e.hill} that the randomness of $ \hill_{j} $ and hence of $ \{\saolambda^{(j)}_i\}_{i=1}^\infty $ comes solely from $ W $.
Our task is to find an `optimal' deviation of $ W $ that realizes the lower bound \eqref{e.lower.goal}.
Namely, we seek a deviation such that, 
when evaluating the l.h.s.\ of \eqref{e.lower.goal} around this deviation one obtains the desired lower bound.
As will be explained in Section~\ref{s.upper}, 
%particularly the first two paragraphs therein,
such a deviation can be chosen to be a constantly drifted \ac{BM}.
That is, we consider the deviation where $ W(y) $ behaves like a constantly drifted \ac{BM} with a drift $ t^{2/3}  v_j \d y $. 
Here $ v_j\in\R $ is a parameter, and the scaling $ t^{2/3} $ matches the aforementioned scaling of $ \lambda $ and $ x $.

We now evaluate the l.h.s.\ of \eqref{e.lower.goal} around the deviation.
Girsanov's theorem asserts that the probability of having such a deviation is
\begin{align}
	\label{e.penality}
	\text{Prob.} \approx \exp(-\tfrac12 t^{a+4/3} v_j^2).
\end{align}
Around such a deviation, Hill's operator behaves like the shifted Laplace operator $ -\frac{\d^2~}{\d y^2} + j \Xi + \frac{2}{\sqrt\beta} t^{2/3} v_j $, 
acting on $ [0,\Xi] $ with the Dirichlet boundary condition.
From this we calculate
\begin{align}
	\notag
	- \sum_{i=1}^\infty (\saolambda^{(j)}_it^{1/3}+zt)_- 
	&=
	- t^{1/3} \int_{(-\infty,-zt^{2/3}]} \d\lambda \, N_j(\lambda) 
\\	
	\notag
	&\approx
	-\frac{t^{1/3}\Xi}{\pi} \int_{-\infty}^{-zt^{2/3}} \d\lambda\, \sqrt{ (\lambda-\tfrac{2}{\sqrt\beta}t^{2/3}v_j-j\Xi)_+ }
\\
	\label{e.cost}
	&=
	-\frac{2t^{a+4/3}}{3\pi} \big( (-z-\tfrac{2}{\sqrt\beta}v_j-jt^{a-2/3})_+ \big)^{3/2}.
\end{align}
Combining~\eqref{e.penality}--\eqref{e.cost} gives
\begin{align}
\label{e.lower.bd}
\begin{split}
	t^{-4/3-a} &\log \Ex\big[ e^{ - \sum_{i=1}^\infty (\saolambda^{(j)}_it^{1/3}+zt)_- } \big] 
\\
	&\gtrsim
	-\tfrac12 v_j^2 - \tfrac{2}{3\pi} (-z-\tfrac{2}{\sqrt\beta}v_j-jt^{a-2/3})_+^{3/2}.
\end{split}
\end{align}
Here $ \gtrsim $ means $ \geq $ with some lower order (in $ t $) error terms. 

The approximate inequality \eqref{e.lower.bd} holds for all $ v_j\in\R $.
It is natural to optimize over $ v_j $.
Differentiating in $ v_j $ shows that the optimum is achieved at
\begin{align*}
	v_{j,*} := 4\pi^{-2}\beta^{-3/2}\big( -1 + \sqrt{1+(\tfrac{\beta\pi}{2})^2(-z-jt^{a-2/3})_+} \big).
\end{align*}
Note that $ v_{j,*}=0 $ for all $ j > -z t^{2/3-a} $, which suggests that we need only to invoke $ j\leq -z t^{2/3-a} $.
With this in mind, we set
\begin{align}
	\label{e.n}
	n= -z t^{2/3-a}.
\end{align}
Sum \eqref{e.lower.bd} over $ j=0,1,\ldots,n-1 $ for $ v_j=v_{j,*} $.
Within the result, the sum over $ j $ can be recognized as a Riemann sum, with $ jn^{a-2/3} $ approximating a continuous variable $ \nu\in[0,\infty) $.
This gives
\begin{align}
\label{e.lower.bd.}
%\begin{split}
%	&\liminf_{t\to\infty} 
%	\frac{1}{t^2} \log \Big( \prod_{j=0}^{n-1} \Ex\big[ e^{ - \sum_{i=1}^\infty (\saolambda^{(j)}_it^{1/3}+zt)_- } \big] \Big)
%\\
	(\text{l.h.s.\ of } \eqref{e.lower.goal})
	\geq
	- \int_0^{\infty} \d \nu \ \tfrac{1}{2} v^2_*(\nu) + \big( (-z-\tfrac{2}{\sqrt\beta}v_*(\nu)-\nu)_+ \big)^{3/2},
%	=:
%	\Psi(v_*),
%\end{split}
\end{align}
where $ v_*(\nu) :=  4\pi^{-2}\beta^{-3/2}(-1 + \sqrt{1+(\tfrac{\beta\pi}{2})^2(-z-\nu)_+}) $.
The integral in \eqref{e.lower.bd.} is explicit and can be evaluated to be $ (2/\beta)^{5} \Phi_-((\beta/2)^2z) $.

We have concluded \eqref{e.lower.goal} for the $ n $ in \eqref{e.n}.
A careful analysis shows that, for such an $ n $, 
the last expectation in~\eqref{e.local.low} is negligible after being taken $ t^{-2}\log(\Cdot) $ and passed to the limit $ t\to\infty $.
The analysis is too involved for the purpose of this article and hence not presented.
The term $ e^{-n} $ in \eqref{e.local.low}  is also negligible because $ t^{-2}\log(e^{-n}) = -t^{-2}(\log(-z)+t^{2/3-a}) \to 0 $ by~\eqref{e.a.range}.
This concludes the discussion of the lower bound.

\section{Upper bound, the WKB condition}
\label{s.upper}

In Section \ref{s.lower}, we utilized a certain type of deviations of $ W $
--- namely constantly drifted \ac{BM} --- to produce the desired lower bound.
To complete the proof, we need to argue that such deviations are optimal, asymptotically as $ t\to\infty $.
We refer to this assertion as the \textbf{WKB condition}.
The terminology is motivated by the fact that our analysis in Sections \ref{s.local}--\ref{s.lower} can be interpreted as the WKB approximation of the stochastic Airy operator.

The WKB condition is by no means obvious in the current context.
To see why, recall that the \ac{BM} enters Hill's operator \eqref{e.hill} through the derivative $ W' $.
Consider the Fourier transform of $ W' $ on the interval $ [0,\Xi] $, i.e., $ W'(y) = \sum_{k\in\Z} W_k e^{2\pi\img y/\Xi} $.
A priori, since $ W' $ is \emph{very rough}, it seems that the high frequency modes $ W_k $, $ k\gg 1 $, could have significant impacts on the spectrum of Hill's operator.
The WKB condition, however, asserts that only the constant mode $ W_0 $ matters for the \ac{LDP} in question. 
Let us further emphasize that the WKB condition may be violated for some other cost functions.
More precisely, here we are concerned with
$
	\Ex[ \exp(\sum_{i=1}^\infty \til{\psi}_t(\hilllambda^{(j)}_i) ],
$
with the cost function $ \til{\psi}_t(\lambda) := -(\lambda t^{1/3}+zt)_- $.
As claimed previously and will be verified in the sequel,
the major contribution of this expectation comes from configurations with constantly drifted $ W $.
On the other hand, we expect that there exists some other cost function $ \hat{\psi}_t $, such that the major contribution of
%\begin{align*}
$
	\Ex[ \exp(\sum_{i=1}^\infty \hat{\psi}_t(\hilllambda^{(j)}_i) ]
$
%\end{align*}
arises from deviations where high frequency modes of $ W $ contribute. 
A class of cost functions that should enjoy the WKB condition have been studied in \cite{krajenbrink19}.

We now return to the task of verifying the WKB condition.
The first step is to argue that, we can replace the Dirichlet boundary condition for $ \hill_j $ with the periodic boundary condition.
This is proven in \cite{tsai18} by utilizing the interlacing of eigenvalues under different boundary conditions.
We do not repeat the technical argument here, and simply \emph{switch} to the periodic boundary condition hereafter.

We will verify the WKB condition at a deterministic level.
To set up the notation, for a real $ f\in C[0,1] $, consider 
\begin{align}
	\label{e.H}
	H := -\tfrac{\d^2~}{\d y^2} + f',
	\quad
	\til{H} := -\tfrac{\d^2~}{\d y^2} + \tfrac{1}{\Xi}(f(\Xi)-f(0))
\end{align}
acting on $ [0,\Xi] $ with the periodic boundary condition. Note that $ f $ does \emph{not} have to be periodic.
The following Proposition encapsulates the WKB condition.
To see how, apply Proposition~\ref{p.wkb} with $ r=t^{2/3}z+j\Xi $ and with $ f=\frac{2}{\sqrt\beta}W $.
One finds that the linear statistics $ -\sum_{i=1}^\infty (t^{1/3}\hilllambda^{(j)}_i+tz)_- $ of $ \hill_j $  
is bounded above by the same linear statistics of an operator with $ W' $ replaced by the average $ \frac{1}{\Xi}(W(\Xi)-W(0)) $. 
This shows that, among all configurations of $ W $ with both ends $ W(0) $ and $ W(\Xi) $ fixed,
the constantly drifted configuration performs the best.

\begin{prop}
\label{p.wkb}
For any real $ f\in C[0,1] $, consider the operators $ H $ and $ \til{H} $ defined in \eqref{e.H}, 
and let $ \{\lambda_{1}\leq\lambda_{2}\leq\ldots\} $ and $ \{\til{\lambda}_{1}\leq\til{\lambda}_{2}\leq\ldots\} $ denote their respective spectra.
For any $ r\in\R $ we have
\begin{align*}
	-\sum_{i=1}^\infty (r+\lambda_i)_- \leq -\sum_{i=1}^\infty (r+\til{\lambda}_i)_-. 
\end{align*}
\end{prop}

\begin{proof}
The first step is to recognize that, for any sequence of real numbers put in ascending order $ -\infty<a_1\leq a_2\leq \ldots $, we have
\begin{align}
	\label{e.var1}
	-\sum_{i=1}^\infty (r+a_i)_-
	=
	\inf_{N} \Big\{ \sum^N_{i=1} \big(a_i+r\big) \Big\}.
\end{align}

We will apply this inequality with $ a_i=\lambda_i $, so for now let us focus on bounding the sum $ \sum^N_{i=1} \lambda_i $, for generic $ N\in\Z_{\geq 0} $.
The operator $ H $ is self-adjoint on $ L^2[0,\Xi] $ (with the periodic boundary condition),
and hence the corresponding eigenvectors $ \phi_1,\phi_2,\ldots $ form an orthonormal basis for $ L^2[0,\Xi] $.
We can thus express the sum of the first $ N $ eigenvalues as $ \sum^N_{i=1} \lambda_i = \sum_{i=1}^N \langle \phi_i, H \phi_i \rangle_{L^2} $.
In fact, the sum can be characterized as the infimum of the same quantity when tested over orthonormal sets:
\begin{align}
	\label{e.var2}
	\sum^N_{i=1} \lambda_i 
	=
	\inf_{\{ \psi_1,\ldots,\psi_N \}} \sum_{i=1}^N \langle \psi_i, H \psi_i \rangle_{L^2},
\end{align}
where the infimum is taken over orthonormal $ \psi_1,\ldots,\psi_N $ in the domain of $ H $.
The assertion~\eqref{e.var2} can be proven by expanding each $ \psi_i $ into a linear combination of $ \{\phi_j\}_{j=1}^\infty $. We do not perform the calculation here.

To bound the r.h.s.\ of~\eqref{e.var2}, we insert a particular orthonormal set from the Fourier basis.
That is, we let $ \psi_{1,*}(y),\ldots,\psi_{N,*}(y) $ be the first $ N $ among
\begin{align*}
	\tfrac{1}{|\Xi|^{1/2}},
	\quad
	\tfrac{1}{|\Xi|^{1/2}} e^{2\pi\img y/|\Xi|},
	\quad
	\tfrac{1}{|\Xi|^{1/2}} e^{-2\pi\img y/|\Xi|},
	\quad
	\tfrac{1}{|\Xi|^{1/2}} e^{4\pi\img y/|\Xi|},
	\ldots.
\end{align*}
From~\eqref{e.var2} we obtain
\begin{align*}
	\sum^N_{i=1} \lambda_i 
	\leq
	\sum_{i=1}^N \langle \psi_{i,*}, H \psi_{i,*} \rangle_{L^2}
	=
	\sum_{i=1}^N \int_0^{\Xi} \big( \d y \,  |\psi_{i,*}'|^2 + \d f(y) \, |\psi_{i,*}|^2 \big),
\end{align*}
where the integral against $ \d f(y) $ is interpreted in the Riemann--Stieltjes sense. 
Noting that $ |\psi_{i,*}|^2 \equiv \frac{1}{\Xi} $ and referring to \eqref{e.H}, we see that the last sum is equal to $ \sum_{i=1}^N \langle \psi_{i,*}, \til{H} \psi_{i,*} \rangle_{L^2} $.
Further, since $ \til{H} $ is just a shifted Laplace operator, the Fourier vectors $ \psi_{i,*} $, $ i=1,\ldots,N $, are the first $ N $ eigenvectors of $ \til{H} $.
Consequently, the last sum is equal to $ \sum_{i=1}^N \til{\lambda}_i $, and therefore
\begin{align}
	\label{e.compare}
	\sum^N_{i=1} \lambda_i \leq \sum_{i=1}^N \til{\lambda}_i.
\end{align}

Now, apply~\eqref{e.var1} with $ a_i=\lambda_i $, use \eqref{e.compare} to bound the result, and apply~\eqref{e.var1} with $ a_i=\til\lambda_i $ in reverse.
This concludes the desired result.
\end{proof}

Proposition~\ref{p.wkb} verifies the WKB condition.
The upper bound can now be proven in the same fashion as the lower bound.

%%%%%%%%%%%%%%%%%%%%%%%%%%%%%%%%%
% References
%%%%%%%%%%%%%%%%%%%%%%%%%%%%%%%%%

\end{document}